\documentclass[a4paper,11pt,titlepage,twoside]{article}

\usepackage{graphicx}
\usepackage[T1]{fontenc} 
\usepackage[utf8]{inputenc}
\usepackage[english]{babel}
\usepackage{soul}
\usepackage{amsfonts}
\usepackage{amsmath}
\usepackage{amsthm}
\usepackage{amssymb}
\usepackage{mathrsfs}
\usepackage[top=5cm, bottom=3cm, left=3.3cm, right=3.3cm]{geometry}
\usepackage{setspace}
\usepackage{afterpage}
\usepackage{extarrows}
\usepackage{fancyhdr}
\usepackage{titlesec}
\usepackage{enumitem} \setlist{nosep}
\usepackage[pdftex,breaklinks,colorlinks,linkcolor=blue,
anchorcolor=blue]{hyperref}

\theoremstyle{definition}
\newtheorem{defin}{Definition}[section]
\theoremstyle{plain}
\newtheorem{theo}[defin]{Theorem}
\newtheorem{lem}[defin]{Lemma}
\newtheorem{pro}[defin]{Proposition}
\newtheorem{cor}[defin]{Corollary}
\theoremstyle{definition}
\newtheorem{exm}[defin]{Example}

\renewcommand{\H}{\mathcal{H}}

\newcommand{\n}[1]{\|#1\|}
\newcommand{\nor}{\|\cdot\|}

\renewcommand{\l}{\langle}
\renewcommand{\r}{\rangle}
\newcommand{\N}{\mathbb{N}}
\newcommand{\Z}{\mathbb{Z}}

\newcommand{\C}{\mathbb{C}}

\newcommand{\pint}{\l\cdot,\cdot\r}
\newcommand{\pin}[2]{\l#1 , #2\r}
\newcommand{\no}{\noindent}
\newcommand{\ol}{\overline}

\newcommand{\mez}{{1/2}}
\newcommand{\spann}{\overline{\text{span}}}

\renewcommand{\phi}{\varphi}
\numberwithin{equation}{section}

\renewcommand\labelenumi{\emph{(\roman{enumi})}}
\renewcommand\theenumi\labelenumi

\titleformat{\section}
{\normalfont\fillast \fontsize{12}{15}\scshape}{\thesection.}{0.8em}{}

\titleformat{\subsection}
{\normalfont\fillast \fontsize{11}{12}\scshape}{\thesubsection.}{0.8em}{}

\begin{document}
	
\thispagestyle{plain}

\begin{center}
	\large
	{\uppercase{\bf On some dual frames multipliers with\\ at most countable spectra}} \\
	\vspace*{0.5cm}
	{\scshape{Rosario Corso}}
\end{center}

\normalsize 
\vspace*{1cm}	

\small 

\begin{minipage}{12.5cm}
	{\scshape Abstract.} A dual frames multiplier is an operator consisting of analysis, multiplication and synthesis processes, where the analysis and the synthesis are made by two dual frames in a Hilbert space, respectively. In this paper we investigate the spectra of some dual frames multipliers giving, in particular, conditions to be at most countable. The contribution extends the results available in literature about the spectra of Bessel multipliers with symbol decaying to zero and of multipliers of dual Riesz bases. 
\end{minipage}

\vspace*{.5cm}

\begin{minipage}{12.5cm}
\end{minipage}

\vspace*{.5cm}

\begin{minipage}{12.5cm}
\end{minipage}

\normalsize

\section{Introduction}

This paper deals with the spectra of dual frames multipliers. 
To explain what is a dual frames multiplier, we first recall that a {\it frame} for a separable Hilbert space $(\H,\pint, \nor)$ is a sequence $\varphi:=\{\varphi_n\}_{n\in \N}$ of vectors of $\H$ such that there exist $A,B>0$ and
\begin{equation}
\label{def_frame}
A\n{f}^2 \leq \sum_{n\in \N} |\pin{f}{\varphi_n}|^2\leq B\n{f}^2, \qquad \forall f\in \H. 
\end{equation}
If $\varphi$ is a frame, then there exists at least a frame $\psi:=\{\psi_n\}_{n\in \N}$ such that 
\begin{equation}
\label{dual}
f=\sum_{n\in \N}\pin{f}{\psi_n}\varphi_n, \qquad \forall f\in \H, 
\end{equation}
and we say that $\varphi$ and $\psi$ are {\it dual frames}\footnote{Equivalently, $\varphi$ and $\psi$ are dual frames if
	$
	f=\sum_{n\in \N} \pin{f}{\varphi_n}\psi_n$ for every $f\in \H$.}. 
If $m=\{m_n\}_{n\in \N}\in \ell^\infty$, i.e. a bounded complex sequence,  $\varphi$ and $\psi$ are dual frames, then the operator $M_{m,\varphi,\psi}$ given by 
\begin{equation}
\label{multipl}
M_{m,\varphi,\psi}f=\sum_{n\in \N} m_n\pin{f}{\psi_n} \varphi_n, \qquad f\in \H
\end{equation}
is called a {\it dual frame multiplier}  
(with {\it symbol}  $m$).

Dual frames $\varphi$ and $\psi$ allow to reconstruct every element  $f\in \H$ by \eqref{dual}, i.e. by an analysis and by a synthesis through $\varphi$ and $\psi$, respectively. A dual frames multiplier is then a process of analysis and synthesis with an intermediate scaling by $m$. For this reason, dual frames multipliers find application in physics \cite{Gazeau}, signal processing \cite{FeiNow,Matz} and acoustics \cite{BHNS,Balazs_appl} (see also \cite{Balazs_surv}). Dual frames multipliers are special cases of {\it Bessel multipliers} which are defined by \eqref{multipl}, but where $\varphi$ and $\psi$ are now {\it Bessel sequences} for $\H$, i.e. sequences of elements of $\H$ satisfying 
\begin{equation}
\label{def_Bess}
\sum_{n\in \N} |\pin{f}{\varphi_n}|^2\leq B_{\varphi}\n{f}^2 \;\text{ and }\; \sum_{n\in \N} |\pin{f}{\psi_n}|^2\leq B_{\psi}\n{f}^2, 
\end{equation}
for some $B_{\varphi},B_{\psi}>0$ and every $f\in \H$. Bessel multipliers have been introduced in \cite{Balazs_basic_mult} and further studied in \cite{Balazs_inv_mult2,Balazs_inv_mult,Detail_mult,Riesz_mult,Dual_mult,Comm_mult}. Multipliers have been investigated also in continuous \cite{Mult_cont} and distributional contexts \cite{Corso_distr_mult}. 

For few classes of Bessel multipliers the structure of the spectra is known in literature. 
In particular, we mention the class of Bessel multipliers $M_{m,\varphi,\psi}$ with $\displaystyle \lim_{n\to +\infty}m_n=0$ (Proposition \ref{lem_comp}) and the class of multipliers $M_{m,\varphi,\psi}$ with $\varphi$ and $\psi$ dual Riesz bases (Proposition \ref{pro_Riesz}). Moreover, in the first case the spectrum is at most countable, while in the second case this happens if and only if $m$ has an at most countable set of limit points\footnote{Throughout the paper, a {\it limit (accumulation) point} for a sequence $m$ indicates the limit of a converging subsequence of $m$. }.  
In this paper we focus on some classes of dual frames multipliers, give descriptions of the spectra and sufficient conditions to ensure that the latter are at most countable. More precisely, the main contributions are the following, which are proved in Sections \ref{sec:1} and \ref{sec:2}, respectively. We denote by $\sigma(T)$, $\sigma_p(T)$, $\sigma_c(T)$, $\sigma_r(T)$ and $\sigma_e(T)$  the {spectrum}, {point spectrum}, {continuous spectrum}, {residual spectrum} and {essential spectrum} of a bounded operator $T$ on $\H$, respectively. 

\begin{theo}
	\label{cor_pert_compact}
	Let $\varphi,\psi$ be dual frames for $\H$ and $m\in \ell^\infty$ a sequence with only finitely many limit points $\{l_1,\dots,l_j\}$. In particular, suppose that
	 $\N=\cup_{k=1}^j I_k$ and, for $k=1,\dots, j$, $|I_k|=\infty$ and $\{m_n:n\in I_k\}$ has only $l_k$ as limit point. Define $m'=\{m_n'\}$ as $m_n'=l_k$ if $n\in I_k$. Then the following statements hold.
	\begin{enumerate}
		\item $\sigma(M_{m,\varphi,\psi})$ is at most countable if and only if $\sigma(M_{m',\varphi,\psi})$ is so.
		\item $\sigma_e(M_{m,\varphi,\psi})=\sigma_e(M_{m',\varphi,\psi})$.
		\item Any point of $\sigma(M_{m,\varphi,\psi})\backslash\sigma_e(M_{m,\varphi,\psi})$ is an isolated eigenvalue with finite multiplicity.
		\item The set of the limit points for $\sigma_p(M_{m,\varphi,\psi})$ is contained in $\sigma_e(M_{m,\varphi,\psi})$.
		\item If $\lambda \notin \sigma_e(M_{m,\varphi,\psi})$, then  $\lambda$ is an eigenvalue of $M_{m,\varphi,\psi}$ if and only if $\ol \lambda$ is an eigenvalue of $M_{m,\varphi,\psi}^*=M_{ m^*,\psi,\varphi}$.  
	\end{enumerate} 
\end{theo}

\begin{theo}
	\label{th_main_ex1}
	Let $\varphi$ be a frame for $\H$ with finite excess\footnote{The excess of $\varphi$ is $e(\varphi)=\sup\{|I|:I\subseteq \N \text{ and } \spann\{\varphi_n\}_{n\in \N\backslash I}=\spann\{\varphi_n\}_{n\in \N}\}$ (see also the beginning of Section \ref{sec:2}).} and $\psi$ is a dual frame of $\varphi$. Let $m\in \ell^\infty$ be a sequence with an at most countable set of limit points $\{l_i\}_{i\in I}$. The following statements hold.  
	\begin{enumerate}
		\item $\sigma(M_{m,\varphi,\psi})$ is at most countable.
		\item If $m$ contains only a finite number of distinct elements, then $\sigma(M_{m,\varphi,\psi})$ is finite and $\sigma(M_{m,\varphi,\psi})=\sigma_p(M_{m,\varphi,\psi})$. 
		\item $\{l_i\}_{i\in I}\subset \sigma(M_{m,\varphi,\psi})$ and  $\sigma(M_{m,\varphi,\psi})\backslash\{l_i\}_{i\in I}=\sigma_p(M_{m,\varphi,\psi})\backslash\{l_i\}_{i\in I}$.
		\item $\sigma_r(M_{m,\varphi,\psi})=\varnothing$.
		\item Each eigenvalue $\lambda$ of $M_{m,\varphi,\psi}$ with $\lambda\notin \{l_i\}_{i\in I}$ is isolated in $\sigma(M_{m,\varphi,\psi})$ and has finite multiplicity.
		\item  In the case where $\sigma(M_{m,\varphi,\psi})$ is infinite, if $\lambda$ is a limit point for $\sigma_p(M_{m,\varphi,\psi})$, then $\lambda\in\{l_i\}_{i\in I}$.
		\item If $\lambda \notin \{l_i\}_{i\in I}$, then $\lambda$ is an eigenvalue of $M_{m,\varphi,\psi}$ if and only if $\ol\lambda$ is an eigenvalue of ${M_{m,\varphi,\psi}}^*=M_{ m^*,\psi,\varphi}$.
	\end{enumerate}
	Suppose, in addition, that $m$ is a real sequence with a finite set of limit points $\{l_i\}_{i\in I}$ and $\psi$ is the canonical dual of $\varphi$. The following statements hold.   
	\begin{enumerate}
		\item[\emph{(viii)}] If $m$ contains infinitely many different elements, then $\sigma(M_{m,\varphi,\psi})$ is countable  and $\{l_i\}_{i\in I}$ is the set of limit points for $\sigma_p(M_{m,\varphi,\psi})$.
	\end{enumerate}
\end{theo}

The topic about the spectra is, of course, more general than that about the invertibility of multipliers, which was the subject of \cite{Balazs_inv_mult2,Balazs_inv_mult,Detail_mult,Riesz_mult,Dual_mult}.
In Theorem \ref{th_inv} (which is used to prove Theorem \ref{th_main_ex1}) we provide a characterization of the invertibility of Bessel multipliers, completing some necessary conditions given in \cite{Balazs_inv_mult}.  

The paper has also a connection with \cite{Bag_Kuz}, which concerns multipliers $M_{m,\phi,\phi}$ where $m$ is real and $\phi$ is a {\it Parseval frame} for $\H$, i.e. a frame for $\H$ for which  \eqref{def_frame} holds with $A=B=1$ (and in that case $\phi$ is a dual of itself).  As seen in Theorem \ref{th_main_ex1} we give a special attention to dual frames multipliers $M_{m,\varphi,\psi}$ involving a real symbol $m$, which is the most interesting case in applications (where  $m$ is sometimes even chosen to be a sequence of $0$ and $1$, see \cite{Balazs_surv}). However, we put the results in more generality, do not assuming necessarily that $\psi=\varphi$.

\section{Preliminaries}
\label{sec:pre}

We write $\H$ to indicate a separable Hilbert space with inner product $\pint$ and norm $\nor$. 
Given an operator $T$ acting between two Hilbert spaces $\H_1$ and $\H_2$, we denote by  $R(T)$ and $N(T)$ the  {\it range} and {\it kernel} of $T$, respectively, and by $T^*$ its {\it adjoint} when $T$ is bounded. 
The notations for the {spectrum}, {point spectrum}, {continuous spectrum}, {residual spectrum} and {essential spectrum} of a bounded operator $T:\H\to \H$ have been given in the introduction. We mention that the essential spectrum of $T$ can be defined as the set of $\lambda \in \C$ such that $R(T-\lambda I)$ is not closed or $\dim N(T-\lambda I)=\dim R(T-\lambda I)^\perp=\infty$ (see \cite{Kato}). Throughout the paper we will make use of some properties of the essential spectrum of an operator, which we summarize below. 

\begin{lem}
	\label{lem_essen}
	Let $T:\H\to \H$ be a bounded operator. The following statements hold. 
	\begin{enumerate}
		\item {\cite[Section IV.5.6]{Kato}} $\sigma_e(T)\subset\sigma(T)$.
		\item {\cite[Theorem IV.5.33]{Kato}} If $\sigma_e(T)$ is at most countable, then $\sigma(T)$ is at most countable and any point of $\sigma(T)\backslash\sigma_e(T)$ is an isolated eigenvalue of $T$ with finite multiplicity. As a consequence, the set of the limit points for $\sigma_p(T)$ is contained in $\sigma_e(T)$. 
		\item {\cite[Theorem IV.5.35]{Kato}} The essential spectrum is invariant under compact perturbation, i.e. if $B$ is a compact operator on $\H$, then $\sigma_e(T+B)=\sigma_e(T)$. 
		\item {\cite[Proposition XI.4.2]{Conway}} $\lambda\in \sigma_e(T)$ if and only if $\ol \lambda \in \sigma_e(T^*)$. As a  consequence, if $ \lambda \notin \sigma_e(T)$, then $\lambda$ is an eigenvalue of $T$ if and only if $\ol \lambda$ is an eigenvalue of $T^*$.
	\end{enumerate}
\end{lem}

We denote by $\ell^2$ (respectively, $\ell^\infty$)  the usual spaces of square summable (respectively, bounded) complex sequences indexed by $\N$. 

We now recall some notions and elementary results about frame theory, that can be found e.g. in \cite{Chris}. 
A sequence $\varphi=\{\varphi_n\}_{n\in \N}$ is {\it complete} in $\H$ if its linear span is dense in $\H$ if and only if $\pin{\varphi_n}{f}=0$ for every $n\in \N$ implies $f=0$.  We make the distinction between the {\it linear span}, denoted by span$(\varphi)$, and the {\it closed linear span}, denoted by $\spann(\varphi)$. 

Let $\varphi$ be a Bessel sequence in $\H$ (the definition has been given in \eqref{def_Bess}). We denote by $C_\varphi:\H\to \ell^2$ the {\it analysis operator} of $\varphi$ defined by $C_\varphi f=\{\pin{f}{\varphi_n}\}$ and by $D_\varphi:\ell^2 \to \H$ the {\it synthesis operator} of $\varphi$ defined by $D_\varphi \{c_n\}=\sum_{n\in \N} c_n \varphi_n$.  It is well-known that these operators are bounded and $D_\varphi=C_\varphi^*$. 

The definitions of frame and Parseval frame for $\H$ can be found in the introduction. We recall that a frame for $\H$ is, in particular, complete in $\H$. 
If $\varphi=\{\varphi_n\}_{n\in \N}$ is a frame for $\H$, then the {\it frame operator} $S_\varphi=D_\varphi C_\varphi:\H \to \H$ is bounded and bijective, $\{S_\varphi^{-1} \varphi_n\}_{n\in \N}$ is a dual frame of $\varphi$, called the {\it canonical dual}, and $\{S_\varphi^{-\mez} \varphi_n\}_{n\in \N}$ is a Parseval frame for $\H$, called the {\it canonical Parseval frame} of $\varphi$.

	A {\it Riesz basis} $\varphi$ for $\H$  is a complete sequence in $\H$ satisfying for some $A,B>0$ 
	
\begin{equation}
\label{def_Riesz}
	A \sum_{n\in \N} |c_n|^2 \leq \left \| \sum_{n\in \N} c_n \varphi_n \right \|^2 \leq B\sum_{n\in \N} |c_n|^2, \qquad\forall \{c_n\}\in \ell^2.
\end{equation}

\no A Riesz basis $\varphi$ for $\H$ is a frame for $\H$ (the constants in \eqref{def_frame} can be chosen as in \eqref{def_Riesz}) and there exists a unique Riesz basis $\psi$ dual to $\varphi$ in the sense of \eqref{dual}. Moreover, a Riesz basis $\phi$ is bounded from below and from above, i.e. $0<\inf_{n\in \N}\|\phi_n\|\leq \sup_{n\in \N}\|\phi_n\|<\infty $.

Finally, given a complex sequence $m=\{m_n\}_{n\in \N}$ and a sequence $\varphi=\{\varphi_n\}_{n\in \N}$ in $\H$, we write  $m^*:=\{\ol m_n\}_{n\in \N}$ (where the bar indicates the complex conjugates of $m_n$)  and $m\varphi:=\{m_n\varphi_n\}_{n\in \N}$.

As mentioned in the introduction, the structure of the spectra is known for two types of multipliers. We summarize the corresponding results available in the literature in the following propositions\footnote{We recall that if $\varphi,\psi$ are Bessel sequences and $m\in \ell^\infty$, then $M_{m,\varphi,\psi}^*=M_{m^*,\psi,\varphi}$ (\cite[Theorem 6.1]{Balazs_basic_mult}).}.

\begin{pro}
	\label{lem_comp}
	Let $\varphi$, $\psi$ be Bessel sequences in  $\H$ and $m\in \ell^\infty$ such that $\displaystyle \lim_{n\to \infty} m_n=0$. Then $M_{m,\varphi,\psi}$ is compact. In particular, 
	\begin{enumerate}
		\item $\sigma(M_{m,\varphi,\psi})$ is at most countable;
		\item $0\in  \sigma(M_{m,\varphi,\psi})$ and  $\sigma(M_{m,\varphi,\psi})\backslash\{0\}=\sigma_p(M_{m,\varphi,\psi})\backslash\{0\}$;
		\item each eigenvalue $\lambda\neq	0$ of $M_{m,\varphi,\psi}$ is isolated in $\sigma(M_{m,\varphi,\psi})$ and has finite multiplicity;
		\item  if $\sigma(M_{m,\varphi,\psi})$ is infinite, then $0$ is the only limit point for $\sigma_p(M_{m,\varphi,\psi})$;
		\item if $\lambda \neq 0$, then $\lambda$ is an eigenvalue of $M_{m,\varphi,\psi}$ if and only if $\ol\lambda$ is an eigenvalue of ${M_{m,\varphi,\psi}}^*$.
	\end{enumerate}
\end{pro}

\begin{pro}
	\label{pro_Riesz}
	Let $\varphi$ be a Riesz basis for $\H$,  $\psi$ its canonical dual and $m\in \ell^\infty$. Then the following statements hold.
	\begin{enumerate}
		\item $\sigma(M_{m,\varphi,\psi})$ is the closure of $m$, considered as set of the complex plane;
		\item the point spectrum of $M_{m,\varphi,\psi}$ is  $\sigma_p(M_{m,\varphi,\psi})=\{m_n:n\in \N\}$;
		\item the continuous spectrum of $M_{m,\varphi,\psi}$ is the set of limit points of $m$;
		\item the residual spectrum of $M_{m,\varphi,\psi}$ is empty. 
	\end{enumerate}
\end{pro}

Proposition \ref{lem_comp} follows by  
\cite[Theorem 6.1]{Balazs_basic_mult} and by the properties of compact operators (\cite[Theorem A.3]{Schm}). 
Proposition \ref{pro_Riesz}(i) is obtained by the fact that $M_{m,\varphi,\psi}-\lambda I=M_{m-\lambda,\varphi,\psi}$, where $m-\lambda:=\{m_n-\lambda\}_{n\in \N}$, and by an application of  
\cite[Theorem 5.1]{Balazs_inv_mult}\footnote{A different proof of Proposition \ref{pro_Riesz} was given in \cite{Corso_seq} (see the comment after Proposition 10).}. The rest of the statement of Proposition \ref{pro_Riesz} follows readily.

Under the hypothesis of Proposition \ref{lem_comp}, $\sigma(M_{m,\varphi,\psi})$ is at most countable and the same holds by Proposition \ref{pro_Riesz} if $\varphi$ is a Riesz basis of $\H$, $\psi$ is the canonical dual of $\varphi$ and $m$ has an at most countable set of limit points. Theorems \ref{cor_pert_compact} and \ref{th_main_ex1} then follow the line of the results in Propositions \ref{lem_comp} and \ref{pro_Riesz}. In particular, the latter is a special case of Theorem \ref{th_main_ex1}.

As final preliminary results we give the following characterizations including one for bijective multipliers (note that previously, in \cite{Balazs_inv_mult2,Balazs_inv_mult}, only some necessary or sufficient conditions for multipliers to be bijective have been given). We write $V+W$ (resp., $V\dotplus W$) for the {\it sum} (resp., {\it direct sum}) of two vector spaces $V,W$.

\begin{theo}
	\label{th_inv}
	Let $\varphi,\psi$ be Bessel sequences in $\H$ and $m\in \ell^\infty$. 
	For each of the following cases \emph{(i-iii)} the statements \emph{(a), (b), ...} below	are equivalent. 
	\begin{enumerate}
		\item \begin{enumerate}[label={\emph{(\alph*)}}]
			\item $M_{m,\varphi,\psi}$ is injective;
			\item $\psi$ is complete and $N(D_{m\varphi})\cap R(C_\psi)=\{0\}$;
			\item  $m\psi$ is complete and $N(D_{\varphi})\cap R(C_{{m^*}\psi})=\{0\}$. 
		\end{enumerate} 
		\item 
		\begin{enumerate}[label={\emph{(\alph*)}}]
			\item $M_{m,\varphi,\psi}$ is surjective;
			\item $\varphi$ is a frame for $\H$ and $N(D_{\varphi})+ R(C_{{m^*}\psi})=\ell^2$;
			\item $m\varphi$ is a frame for $\H$ and $N(D_{m\varphi})+ R(C_\psi)=\ell^2$.
		\end{enumerate}
		\item 
		\begin{enumerate}[label={\emph{(\alph*)}}]
			\item $M_{m,\varphi,\psi}$ is bijective;
			\item $\varphi$ is a frame for $\H$ and $N(D_{\varphi})\dotplus R(C_{m^*\psi})=\ell^2$;
			\item $m\varphi$ is a frame for $\H$ and $N(D_{m\varphi})\dotplus R(C_\psi)=\ell^2$;
			\item $\psi$ is a frame for $\H$ and $N(D_{\psi})\dotplus R(C_{m\varphi})=\ell^2$;
			\item $m\psi$ is a frame for $\H$ and $N(D_{m^*\psi})\dotplus R(C_\varphi)=\ell^2$. 
		\end{enumerate} 
	\end{enumerate}
\end{theo}
\begin{proof}
	\begin{enumerate}[label={(\roman*)}]
		\item Since $M_{m,\varphi,\psi}=D_{m\varphi} C_\psi=D_{\varphi} C_{m^*\psi}$, the equivalence of the statements follows from the fact that $f\in N(M_{m,\varphi,\psi})$ if and only if $C_\psi f \in N(D_{m\varphi})$ if and only if $C_{m^*\psi} f \in N(D_{\varphi})$. 
		\item First we assume that $M_{m,\varphi,\psi}$ is surjective. Writing $M_{m,\varphi,\psi}=D_{\varphi}C_{m^*\psi}$ we find that $D_{\varphi}$ is surjective, i.e. $\varphi$ is a frame for $\H$ by \cite[Theorem 5.5.1]{Chris}.
		Moreover, let $\{c_n\}\in \ell^2$ and let $g=D_\varphi\{c_n\}$. Thus there exists $f\in \H$ such that $g=M_{m,\varphi,\psi}f=D_{\varphi}C_{m^*\psi} f$. This implies that $\{c_n\}-C_{m^*\psi} f$ belongs to $N(D_{\varphi})$ and therefore $N(D_{\varphi})+ R(C_{m^*\psi})=\ell^2$. \\
		Now suppose that (b) holds. Let $g\in \H$. There exists $\{c_n\}\in \ell^2$ such that $D_{\varphi} \{c_n\}=g$ again by \cite[Theorem 5.5.1]{Chris}. Furthermore, there exist $\{d_n\}\in N(D_{\varphi})$ and $f\in \H$ such that $\{c_n\}=\{d_n\}+C_{m^*\psi} f$. Therefore, $g=D_{\varphi} C_{m^* \psi} f=M_{m,\varphi,\psi} f$, i.e. $M_{m,\varphi,\psi}$ is surjective. The equivalence of (a) and (c) can be proved similarly. 
		\item It follows by (i) and (ii). \qedhere
	\end{enumerate}
\end{proof}

\section{Proof of Theorem \ref{cor_pert_compact}}
\label{sec:1}

In Theorem \ref{cor_pert_compact}, which we are going to prove, we analyze the first class of dual frames multipliers, characterized by symbols $m$ with a finite number of limit points. The assumption that $m$ has a finite set of limit points will be present also later and, of course, is not restrictive in applications because of practical reasons. 
What the theorem says is that to decide whether the spectrum of a dual frames multiplier $M_{m,\varphi,\psi}$ is at most countable it is sufficient to check the spectrum of a related multiplier $M_{m',\varphi,\psi}$, where $m'$ is the union of constant sequences.

\begin{proof}[Proof of Theorem \ref{cor_pert_compact}]
	The statement is directly proved noting that
	$M_{m,\varphi,\psi}$ is a compact perturbation of $M_{m',\varphi,\psi}$, so Lemma \ref{lem_essen} applies. Indeed, $M_{m,\varphi,\psi}=M_{m',\varphi,\psi}+M_{m'',\varphi,\psi}$ where $m''$ is a sequence converging to $0$ and hence $M_{m'',\varphi,\psi}$ is compact by Proposition \ref{lem_comp}. 
\end{proof}

\begin{exm}
	Let  $\varphi=\{\frac{1}{\sqrt{2}}e_1,\frac{1}{\sqrt{2}}f_1, \frac{1}{\sqrt{2}}e_2, \frac{1}{\sqrt{2}}f_2, \dots\}$ where $\{e_n\}$ and $\{f_n\}$ are orthonormal bases for $\H$. A simple computation shows that $\varphi$ is a Parseval frame for $\H$. Now let $m=\{m_n\}$ be such that $m_{2n-1}=\frac{1}{n+1}$ and $m_{2n}=2-\frac{1}{n+1}$, $n\in \N$. Then,
	Theorem \ref{cor_pert_compact} applies to $M_{m,\varphi,\varphi}$, indeed $M_{m,\varphi,\varphi}$ is a compact perturbation of the multiplier $M_{ m',\varphi,\varphi}$ where $m'=\{m_n'\}$ and $m_{2n-1}'=0$, $m_{2n}'=2$, $n\in \N$.  More precisely, since $M_{ m',\varphi,\varphi}=I$, we can say that $\sigma(M_{m,\varphi,\varphi})$ is an at most countable set, $1\in \sigma(M_{m,\varphi,\varphi})$ and  $\sigma(M_{m,\varphi,\varphi})\backslash\{1\}$ consists of isolated eigenvalues with finite multiplicities which have $1$ as limit point if $\sigma(M_{m,\varphi,\varphi})$ is infinite.
\end{exm}

As particular case of Theorem \ref{cor_pert_compact} we get the following extension of Proposition \ref{lem_comp}. 

\begin{cor}
	Let $\varphi$, $\psi$ be dual frames for $\H$ and $m\in \ell^\infty$ such that $\displaystyle \lim_{n\to \infty} m_n=l$. Then the following statements hold. 
	\begin{enumerate}
		\item $\sigma(M_{m,\varphi,\psi})$ is at most countable.
		\item $l\in  \sigma(M_{m,\varphi,\psi})$ and  $\sigma(M_{m,\varphi,\psi})\backslash\{l\}=\sigma_p(M_{m,\varphi,\psi})\backslash\{l\}$.
		\item Each eigenvalue $\lambda$ of $M_{m,\varphi,\psi}$ with $\lambda\neq l$ is isolated in $\sigma(M_{m,\varphi,\psi})$ and has finite multiplicity.
		\item  If $\sigma(M_{m,\varphi,\psi})$ is infinite, then $l$ is the limit point for $\sigma_p(M_{m,\varphi,\psi})$.
		\item If $\lambda \neq l$, then $\lambda$ is an eigenvalue of $M_{m,\varphi,\psi}$ if and only if $\ol\lambda$ is an eigenvalue of ${M_{m,\varphi,\psi}}^*=M_{ m^*,\psi,\varphi}$.
	\end{enumerate}
\end{cor}

\section{Proof of Theorem \ref{th_main_ex1}}
\label{sec:2}

The second class of multipliers with at most countable spectra we analyze concerns dual frames with finite excesses. We recall that the {\it excess} of a sequence $\varphi$ is 
$$e(\varphi)=\sup\{|I|:I\subseteq \N \text{ and } \spann\{\varphi_n\}_{n\in \N\backslash I}=\spann\{\varphi_n\}_{n\in \N}\},$$ 
and for a frame $\varphi$, the excess  $e(\varphi)$ is equal to $\dim N(D_\varphi)=\dim R(C_\varphi)^\perp$, see \cite{BCHL_excess,Holub} (thus a frame has null excess if and only if it is a Riesz basis). 
 Note also that dual frames have the same excess (\cite[Theorem 2.2]{Beric}). 
 We will make use of the following result.
 
\begin{pro}
	\label{pro_phi_mphi}
	Let $\varphi$ be a frame for $\H$ and $m$ a complex sequence such that also $m\varphi$ is a frame for $\H$. Then $\varphi$ and $m\varphi$ have the same excess. 
\end{pro}
\begin{proof}
	Let $J=\{n\in \N: m_n=0\}$. Since 
	\begin{equation}
	\label{excesss}
	\H=\spann(m\varphi)=\spann(\{m_n\varphi_n\}_{n\in \N\backslash J })=\spann(\{\varphi_n\}_{n\in \N\backslash J })
	\end{equation}
	it follows that $|J|\leq e(\varphi)$. Therefore, if $|J|=\infty$, then both the excesses of $\varphi$ and $m\varphi$ are infinity. On the other hand, if $|J|<\infty$ we can say that $\{\varphi_n\}_{n\in \N\backslash J }$ is a frame for $\H$ with excess $e(\varphi)-|J|$ (see \cite[Theorem 5.4.7]{Chris}). Hence, by \eqref{excesss}, $e(m\varphi)=|J|+(e(\varphi)-|J|)=e(\varphi)$. 	
\end{proof}

We need some other preparatory results before to prove Theorem \ref{th_main_ex1}. Let $\varphi$ be a frame for $\H$ with finite excess $q$, $\psi$ a dual of $\varphi$ and $m\in \ell^\infty$. By \cite[Theorem 2.4]{Holub},  $\N=I\cup J$ with $I\cap J=\varnothing$, $|I|=q$ and such that $\{\varphi_n\}_{n\in J}$ is a Riesz basis for $\H$. We write $I=\{n_1,n_2,\dots,n_q\}$. As a consequence, for every $i=1,\dots,q$, we can find a unique $d_i\in N(D_\varphi)$ such that $d_i^{n_k}=\delta_{n_i,n_k}$ for all $k=1,\dots,q$. Moreover, $d_1,\dots, d_q$ form a spanning set for $N(D_\varphi)$. 

Under the assumption that $\inf_{n\in \N} |m_n-\lambda|>0$ it is clear that a spanning set of $N(D_{(m-\lambda) \varphi})$ is $\{u_1,\dots,u_q\}$ where 
\begin{equation}
\label{spanning_set}
u_i=\{u_i^n\}_{n\in \N}, \qquad u_i^n= \frac{d_i^n}{m_n-\lambda}, \qquad \forall  i=1,\dots,q, \;\forall n\in \N. 
\end{equation}

Now we define a $q\times q$ matrix as follows 
\begin{equation}
\label{matr_spec}
\mathcal{A}_\lambda=(\pin{u_i}{v_j}_2)_{i,j=1,\dots,q}
\end{equation}
where $\{v_1,\dots,v_q\}$ is a spanning set of $N(D_\psi)$ and $\pint_2$ is the inner product of $\ell^2$. The matrix $\mathcal{A}_\lambda$ depends on the choice of $d_i\in N(D_\varphi)$ and $v_i\in N(D_\psi)$. However, the property of $\det(\mathcal{A}_\lambda)$ of being $0$ does not depend on the choice of $d_i\in N(D_\varphi)$ and $v_i\in N(D_\psi)$. Indeed, $\det(\mathcal{A}_\lambda)=0$ holds if and only if there exists $u\in N(D_{(m-\lambda) \varphi})$, $u\neq 0$, such that $u\in N(D_\psi)^\perp$, i.e. if and only if $N(D_{(m-\lambda) \varphi}) \cap R(C_\psi)\neq \{0\}$.

\begin{lem}
	\label{lem_spec_ex}
	Let $\varphi$ be a frames for $\H$ with finite excess, $\psi$ is a dual frame of $\varphi$ and $m\in \ell^\infty$. The spectrum $\sigma(M_{m,\varphi,\psi})$ of the multiplier $M_{m,\varphi,\psi}$ is the set of $\lambda\in \C$ such that one of the following statements is satisfied:
	\begin{enumerate} 
		\item[\emph{(i)}] $\lambda$ is a limit point for $m$;
		\item[\emph{(ii)}] $\lambda$ is not an  limit point for $m$, $\lambda\in m$ and $(m-\lambda)\varphi:=\{(m_n-\lambda)\varphi_n\}_{n\in \N}$ is not complete in $\H$;
		\item[\emph{(iii)}] $\lambda$ is not an  limit point for $m$, $\lambda\in m$ and $\ell^2$ is not the direct sum of $N(D_{(m-\lambda)\varphi})$ and $R(C_\psi)$;
		\item[\emph{(iv)}]  $\inf_{n\in \N} |m_n-\lambda|>0$ and 
		$\det( \mathcal{A}_\lambda)= 0$. 		
	\end{enumerate}
	Furthermore, 
	\begin{enumerate}
		\item[\emph{(a)}] if $\lambda$ satisfies \emph{(i)} and $\lambda=m_n$ for infinitely many $n$, then $\lambda \in \sigma_p(M_{m,\varphi,\psi})$;
		\item[\emph{(b)}] if $\lambda$ satisfies \emph{(i)} and $\lambda=m_n$ for only finitely many $n$, then either $\lambda \in \sigma_p(M_{m,\varphi,\psi})$ or $\lambda \in \sigma_c(M_{m,\varphi,\psi})$;
		\item[\emph{(c)}] if $\lambda$ satisfies \emph{(ii)}, \emph{(iii)} or \emph{(iv)}, then $\lambda \in \sigma_p(M_{m,\varphi,\psi})$.
	\end{enumerate}
	 
\end{lem}
\begin{proof}
	We divide the proof in steps. 
	\begin{itemize}
		\item $\lambda$ satisfies (i) $\implies$ $\lambda\in \sigma(M_{m,\varphi,\psi})$. \\ Let $\lambda$ be a limit point for $m$. 
		By the discussion preceding the statement we can suppose, without loss of generality, that $\psi':=\{\psi_{q+1},\psi_{q+2},\dots\}$ is a Riesz basis for $\H$. We denote by $\widetilde\psi'=\{\widetilde\psi'_n\}$ the canonical dual of $\psi'$. Let $\{m_{n_k}\}$ be a subsequence of $\{m_n\}$ such that $m_{n_k}\to \lambda$. The sequence $\{\widetilde\psi'_{n_k}\}$ is contained into a Riesz basis, so it is bounded from below and above; however
		\begin{equation}
		\label{range_notclosed}
		(M_{m,\varphi,\psi}-\lambda I) \widetilde\psi'_{n_k}=\sum_{i=1}^q (m_i-\lambda)\pin{\widetilde\psi'_{n_k}}{\psi_i}\varphi_i+(m_{n_k}-\lambda)\varphi_{n_k}\to 0,
		\end{equation}
		as $k\to \infty$. Therefore, $\lambda \in \sigma(M_{m,\varphi,\psi})$.
		\item (a) holds. \\ 
		If $\lambda\in m$ is a limit point for $m$ and $\lambda=m_n$ for infinitely many $n$, then infinitely many elements in $(m-\lambda)\varphi$ are zero. So $(m-\lambda)\varphi$ cannot be complete in $\H$ (recall that $\varphi$ has finite excess) and Theorem \ref{th_inv}(i) implies that $\lambda \in \sigma_p(M_{m,\varphi,\psi})$. 
		\item (b) holds. \\  Let $\lambda$ be a limit point for $m$ such that $\lambda=m_n$ for only finitely many $n$. If $M_{m,\varphi,\psi}-\lambda I$ is not injective, then $\lambda \in \sigma_p(M_{m,\varphi,\psi})$ by definition.   In the other case, i.e. if $M_{m,\varphi,\psi}-\lambda I$ is injective, then  \eqref{range_notclosed} shows that $\lambda \in \sigma_c(M_{m,\varphi,\psi})$.
		\item $\lambda$ satisfies (ii) $\implies$ $\lambda \in \sigma_p(M_{m,\varphi,\psi})$.\\
		Let $\lambda \in m$ be such that it is not a limit point for $m$ and $(m-\lambda)\varphi$ is not complete in $\H$.  Thus, by Theorem  \ref{th_inv}(i), $\lambda\in \sigma_p(M_{m,\varphi,\psi})$. 
		\item $\lambda$ satisfies (iii) $\implies$ $\lambda \in \sigma_p(M_{m,\varphi,\psi})$.\\ 
		Let $\lambda \in m$ be such that it is not a limit point for $m$ and $\ell^2$ is not the direct sum of $N(D_{(m-\lambda)\varphi})$ and $R(C_\psi)$. First of all, we notice that $(m-\lambda)\varphi$ contains at most a finite number of zero elements. So if this sequence is not a frame for $\H$, then it is not complete in $\H$ by \cite[Theorem 5.4.7]{Chris}, i.e. the previous case (ii) applies. Thus, we can confine ourself to the case where $(m-\lambda)\varphi$ is a frame for $\H$. Taking into account  Proposition \ref{pro_phi_mphi} and that $\phi$ and $\psi$ have the same excess, 
		$\dim N(D_{(m-\lambda)\varphi})=\dim N(D_{\varphi})=\dim N(D_{\psi})=\dim R(C_\psi)^\perp<\infty$. Since $\ell^2$ is not the direct sum of $N(D_{(m-\lambda)\varphi})$ and $R(C_\psi)$, we necessarily have   $N(D_{(m-\lambda)\varphi})\cap R(C_\psi)\neq \{0\}$, i.e. $\lambda\in \sigma_p(M_{m,\varphi,\psi})$ again by Theorem  \ref{th_inv}(i).  
		\item $\lambda$ satisfies (iv) $\implies$ $\lambda \in \sigma_p(M_{m,\varphi,\psi})$. \\ Let us assume that $\inf_{n\in \N} |m_n-\lambda|>0$ and 
		$\det( \mathcal{A}_\lambda)= 0$. The first condition ensures that $(m-\lambda)\varphi$ is a frame for $\H$. The latter condition
		is equivalent to $N(D_{(m-\lambda)\varphi})\cap R(C_\psi)\neq \{0\}$ (see the comment right before the statement). 
		Thus, $\lambda \in \sigma_p(M_{m,\varphi,\psi})$, by Theorem  \ref{th_inv}(i). 
		\item $\lambda\in \sigma(M_{m,\varphi,\psi})$ $\implies$ (i), (ii), (iii) or (iv) hold. \\
		Let  $\lambda\in \sigma(M_{m,\varphi,\psi})$. Then there are only three possibility: $\lambda$ is a limit point for $m$, $\lambda$ is not a limit point for $m$ but $\lambda\in m$, or $\inf_{n\in \N} |m_n-\lambda|>0$.  The first case corresponds to (i). In the second case, by Theorem  \ref{th_inv}(iii) $\ell^2$ is not the direct sum of $N(D_{(m-\lambda)\varphi})$ and $R(C_\psi)$ (which corresponds to (iii)) or $(m-\lambda)\varphi$ is not a frame for $\H$. The latter subcase implies that $(m-\lambda)\varphi$ is not complete in $\H$ (i.e., statement (ii)) by \cite[Theorem 5.4.7]{Chris}, since $(m-\lambda)\varphi$ contains at most a finite number of zero elements (recall that we are considering $\lambda\in m$, but not a limit point for $m$). Finally, let us assume that $\inf_{n\in \N} |m_n-\lambda|>0$. Since $(m-\lambda)\varphi$ is a frame for $\H$, by Theorem \ref{th_inv}(iii) we necessarily have that  $\ell^2$ is not the direct sum of $N(D_{(m-\lambda)\varphi})$ and $ R(C_\psi)$. Since $\dim N(D_{(m-\lambda)\varphi})=\dim R(C_\psi)^\perp<\infty$ by Proposition \ref{pro_phi_mphi} and the fact that $\phi$ and $\psi$ have the same excess, we have that $N(D_{(m-\lambda)\varphi})\cap R(C_\psi)\neq \{0\}$ , i.e.  $\det(\mathcal{A}_\lambda)=0$ (statement (iv)). 
	\end{itemize}
\end{proof}

Now we are able to prove the main result of this section regarding the multipliers of dual frames with finite excess.

\begin{proof}[Proof of Theorem \ref{th_main_ex1}]
		Statements (ii), (iii) and (iv) hold by Lemma \ref{lem_spec_ex} (note that if $m$ contains only a finite number of distinct elements, then the zeros of $\det(\mathcal{A}_\lambda)$ are the zeros of a polynomial and $\sigma_c(M_{m,\varphi,\psi})=\varnothing$ by parts (a), (b) and (c) of Lemma \ref{lem_spec_ex}). 
		 
		To demonstrate statements (i), (v) and (vi), we first prove that the essential spectrum $\sigma_e(M_{m,\varphi,\psi})$ is the set of limit points for $m$.  
		By \cite[Theorem 1.2]{hlll} there exist a Hilbert space $\mathcal K$ such that $\H\subset \mathcal K$, an orthogonal projection $P:\mathcal K\to \H$ and a Riesz basis $\xi:=\{\xi_n\}$ of $\mathcal K$ with dual $\widetilde \xi:=\{\widetilde\xi_n\}$ such that $\varphi_n=P \xi_n$ and $\psi_n=P \widetilde \xi_n$ for every $n\in \N$. Then $M_{m,\varphi,\psi}$ can be identified with the restriction of $M_{m,P \xi,P \widetilde \xi}$ to $\H$, so in particular the essential spectra of the two operators are the same. Writing $M_{m,P \xi,P \widetilde \xi}=M_{m,\xi,\widetilde \xi}-M_{m,P \xi,(I-P)\widetilde \xi}-M_{m,(I-P)\xi,P\widetilde \xi}-M_{m,(I-P)\xi,(I-P)\widetilde \xi}$, we note that $M_{m,P\xi,P\widetilde \xi}$ is a finite rank perturbation of $M_{m,\xi,\widetilde \xi}$ (indeed $I-P$ is a projection onto a subspace with dimension $q$, the excess of $\varphi$). Therefore, by Lemma \ref{lem_essen},  $\sigma_e(M_{m,\varphi,\psi})$ is equal to $\sigma_e(M_{m,\xi,\widetilde \xi})$, i.e. the set of limit points for $m$. 
		Now statements (i) and (v) hold by Lemma \ref{lem_essen}(ii) and statement (vi) follows by statements (iii) and (v)\footnote{We remark that also statement (iii) can be proved with the help of Lemma \ref{lem_essen}.}.
		
		Statement (vii) is proved making use again of Lemma \ref{lem_essen}. 
		The proof of statement (viii) follows by Theorem of \cite{Behncke} (see also Corollary 1 of \cite{Behncke}). In more details, assume that $m$ is a real sequence with a finite set of limit points $\{l_i\}_{i\in I}$ and contains infinitely many different elements. Assume also that $\psi$ is the canonical dual of $\varphi$. Since $M_{m,\varphi,\psi}$ is similar\footnote{Indeed, if $S$ is the frame operator of $\varphi$, then $\rho=S^{-\mez}\varphi$ is the Parseval frame of $\varphi$ and $S^{-\mez}M_{m,\varphi,\psi}S^{\mez}=M_{m,\rho,\rho}$.} to $M_{m,\rho,\rho}$ where $\rho$ is the canonical Parseval frame of $\varphi$ and the spectrum is preserved under similarity, we can confine ourselves to the case where $\varphi=\psi$ is a Parseval frame for $\H$. Hence, $\xi$ can be taken equal to an orthonormal basis for $\mathcal{K}$ and $M_{m,\varphi,\varphi}$ can be identified with the restriction of $M_{m,P\xi,P\xi}$ to $\H$. 
		Let $[\alpha,\beta]$ be an interval which does not contain any $l_i$. 	 
			We denote by $\{m_k'\}$ the elements of $m$ in $[\alpha,\beta]$ counted without repetitions and let $r_k$ be the multiplicity of $m_k'$ in $m$. 
			Since both $M_{m,e,e}$ and $M_{m,P\xi,P\xi}$ are self-adjoint operators, by \cite[Theorem]{Behncke} there are at least $\sum_{k} r_k-3q$  
			eigenvalues of $M_{m,P\xi,P\xi}$ in $[\alpha,\beta]$. 
			By a limit procedure, considering an increasing subsequence $\{m_n'\}_{n\in \N}$ converging to some $l_i$ and $\{m_n'\}_{n\in \N}\subset [\alpha,\beta]$ (where $[\alpha,\beta]$ does not contain limit points for $m$ different to $l_i$) we obtain that $\sigma_p(M_{m,P\xi,P\xi})$ is infinite and $l_i$ is a limit point for the eigenvalues of $M_{m,P\xi,P\xi}$. Now the conclusion follows noting that  $\sigma(M_{m,\varphi,\varphi})\backslash\{0\}=\sigma_\mathcal{K}(M_{m,P\xi,P\xi})\backslash\{0\}$\footnote{$\sigma_\mathcal{K}(M_{m,P\xi,P\xi})$ denotes the spectrum of $M_{m,P\xi,P\xi}$ as operator in $\mathcal{K}$.}. 
\end{proof}

By Lemma \ref{lem_spec_ex} and Theorem \ref{th_main_ex1} we can add that if $\varphi$ is a frame for $\H$ with finite excess, $\psi$ is a dual frame of $\varphi$ and $m\in \ell^\infty$,  then
$\sigma(M_{m,\varphi,\psi})$ is at most countable if and only if $m$ has an at most countable number of limit points. We can add a further consideration about the speed of convergence of eigenvalues which generalizes Proposition 3.3 in \cite{Balazs_surv} about the case of Schatten $p$-classes multipliers.

\begin{cor}
	Let $\varphi$, $\psi$ be a frame for $\H$ with finite excess and its canonical dual. Let $m\in \ell^\infty$ be a real sequence with a finite set of limit points $\{l_1,\dots,l_j\}$. Suppose that $m$ is a disjoint union $m=\cup_{i=1}^j\{m_{i,n}:n\in \N\}$ and for each $i=1,\dots,j$ there exists $p_i>1$ such that $\{m_{i,n}-l_i:n\in \N\}\in \ell^{p_i}$. Then $\sigma_p(M_{m,\varphi,\psi})=\cup_{i=1}^j\{\lambda_{i,n}:n\in \N\}$ such that for each $i=1,\dots,j$ one has $\{\lambda_{i,n}-l_i:n\in \N\}\in \ell^{p_i}$. 
\end{cor}
\begin{proof}
	Let $m$ be as in the statement. Fix $i=1,\dots,j$ and write $\{m_{i,n}^l\}=\{m_{i,n}:n\in \N\}\cap (l_{i-1}, l_i]$ and  $\{m_{i,n}^r\}=\{m_{i,n}:n\in \N\}\cap [l_i,l_{i+1})$, where $l_{0}:=-\infty$ and $l_{j+1}:=\infty$. Now take the elements of  $\{m_{i,n}^l\}$ without repetitions and order them in a increasing way $m_1'<m_2'<\dots< m_k'<\dots$. Let $r_k$ be the multiplicity of $m_k'$ in $\{m_{i,n}^l\}$. Thinking as seen in the proof of Theorem \ref{th_main_ex1} and applying \cite[Theorem]{Behncke}, we get that in the interval $(m_k',m_{k+1}')$ there are at maximum $3q$  eigenvalues of $M_{m,\varphi,\varphi}$ repeated with multiplicity, where $q$ is the excess of $\varphi$. For the same reason, the multiplicity of an eventual eigenvalue $m_k$ is less or equal than $r_k+3q$. 
	
	Therefore, if $\Lambda_i^l=\sigma_p(M_{m,\varphi,\varphi})\cap [\min_n(m_{i,n}^l),l_i]$ (where the eigenvalues are counted with multiplicity), 
	\begin{align}
	\sum_{\lambda \in \Lambda_i^l} |\lambda-l_i|^{p_i}&\leq \sum_{k} (r_k+3q)|m_{k}'-l_i|^{p_i}+3q  \sum_{k} |m_{k}'-l_i|^{p_i}\\
	&\leq (1+6q) \sum_{n}  |m_{i,n}^l-l_i|^{p_i}<\infty. 
	\end{align}
	An analogous consideration can be applied to $\{m_{i,n}^r\}$. 
	In conclusion, $\sigma_p(M_{m,\varphi,\varphi})\cap [\min_n(m_{i,n}^l),\max_n(m_{i,n}^r)]$ can be written as a sequence $\{\lambda_{i,n}:n\in\N\}$ such that $\lambda_{i,n}\to l_i$ in $\ell^{p_i}$ for $n\to \infty$. 
\end{proof}

We now make some auxiliary considerations through examples.

\begin{exm}
	For a multiplier $M_{m,\varphi,\varphi}$ where $\varphi$ is a Parseval frame for $\H$ with infinite excess and $m=\{1,0,1,0,\dots\}$ at least one of the conditions $\sigma(M_{m,\varphi,\varphi})$ uncountable, $\sigma_p(M_{m,\varphi,\varphi})=\varnothing$ and $|\sigma_p(M_{m,\varphi,\varphi})|=\infty$ is possible. In other words, statements (i) and (ii) of Theorem \ref{th_main_ex1} may not hold for frames with infinite excess. Indeed, we can consider the following examples where $m$ is always taken as $m=\{1,0,1,0,\dots\}$. 
	\begin{enumerate}
		\item[(a)] 	Let $\H=L^2(0,1)$ and $e_n(x)=e^{2\pi i n x}$ for $n\in \Z$ and $x\in (0,1)$. Now define $\varphi_{2n+1}(x)=xe^{2\pi i n x}$ and $\varphi_{2n}(x)=(1-x^2)^\mez e^{2\pi i n x}$. With a simple calculation one can see that $\varphi:=\{\varphi_n\}_{n\in \Z}$ is a Parseval frame for the classical space $L^2(0,1)$. 
		Moreover, $M_{m,\varphi,\varphi}$ is the multiplication operator by $x^2$, thus $\sigma(M_{m,\varphi,\varphi})=[0,1]$ and  $\sigma_p(M_{m,\varphi,\varphi})=\varnothing$. 
		\item[(b)] Let $\H$ be a Hilbert space and $\{e_n\}_{n\in \N}$ an orthonormal basis for $\H$. Define a Parseval frame $\varphi=\{\varphi_n\}_{n\in \N}$ for $\H$ as $\varphi_{2n-1}=\frac {1}{n}  e_n$ and  $\varphi_{2n}=\sqrt{ 1-\frac{1}{n^2}} e_n$, $n\in \N$. We have $\sigma_p(M_{m,\varphi,\varphi})=\{\frac 1 n:n\in \N \}$. 
	\end{enumerate}
		
\end{exm}

\begin{exm}
	Let $\{e_n\}_{n\in \N}$ be an orthonormal basis for $\H$ and $\varphi=\{\varphi_n\}$ as
	$\varphi_{2n-1}=\varphi_{2n}=\frac{1}{\sqrt{2}}e_n$ (a Parseval frame for $\H$). Let $m\in \ell^\infty$. Then $\sigma(M_{m,\varphi,\varphi})$ contains $\tau:=\{\frac12(m_{2n-1}+m_{2n}):n\in \N\}$ and the limit points for $\tau$. In conclusion, statements (iii), (v) and (vi) of Theorem \ref{th_main_ex1}  may not hold if $\varphi$ has infinite excess. This example shows also that $\sigma(M_{m,\varphi,\varphi})$ may be a singleton even though $m$ has an infinite number of limit points. 
\end{exm}

\begin{exm}
	When $m$ contains infinite distinct elements but it is not a real sequence, then  $\sigma(M_{m,\varphi,\psi})$, with $\varphi$ a frame with finite excess and $\psi$ its canonical dual, may not be infinite (so statement (viii) of Theorem \ref{th_main_ex1} may not hold). We give an example of this situation where $\sigma(M_{m,\varphi,\psi})$ contains even a single point. 
	
	Let $\varphi$ be a Parseval frame for $\H$ with excess $1$ and $N(D_\varphi)=\text{span}(d)$ where $d=\{d^n\}$ is defined by 
	$d^1=\frac{1}{\sqrt{2(e+1)}}$, $d^2=\frac{1}{\sqrt{\pi^2+1}}$ and, for $k\in \N$,  $d^{2k+1}=\frac{1}{\sqrt{(2k+1)^2\pi^2+1}}$, $d^{2k+2}=\frac{1}{\sqrt{(-2k+1)^2\pi^2+1}}$ (for instance we can take the Parseval frame associated to $\{-\sum_n d^n e_n ,d^1e_1,\dots,d^1e_n, \dots\}$ with $\{e_n\}$ an orthonormal basis for $\H$). Let $m=\{m_n\}_{n\in \N}$ be defined as follows: $m_1=\frac 12 $, $m_2= 1-\frac{1}{1+\pi i}$, $m_{2k+1}=1-\frac{1}{1+(2k+1)\pi i }$ and $m_{2k+2}=1-\frac{1}{1+(-2k+1)\pi i }$ for $k\in \N$. 
	
	By \cite[Theorem 6.1]{Balazs_basic_mult},  $\sigma(M_{m,\varphi,\varphi}) \subset B_1$, the disk centered in the origin and with radius $1$,  and 
	Lemma \ref{lem_spec_ex} ensures that $1\in \sigma(M_{m,\varphi,\varphi})$. If there were $\lambda \in \sigma(M_{m,\varphi,\varphi})\backslash \{1\}$, then  $\lambda$ would satisfy Lemma \ref{lem_spec_ex}(iv) (indeed   $\{(m_n-\lambda)\varphi_n\}$ is complete and $\ell^2=N(D_{\{(m_n-\lambda)\varphi_n\}})\dotplus R(C_\varphi)$ if $\lambda \in m$), i.e. $\lambda \notin m$ and $\det (\mathcal{A}_\lambda)=0$. In particular, taking $u$ as in \eqref{spanning_set} and $v=d$ in \eqref{matr_spec}, the expression $\det (\mathcal{A}_\lambda)=0$ means 
	\begin{equation*}
	\sum_{n\in \N} \frac{d_n^2}{\lambda -m_n}=0.
	\end{equation*}
	Nevertheless, as shown in \cite{Langley} the function $\lambda \to \sum_{n\in \N} \frac{d_n^2}{\lambda -m_n}$ has no zeros in $B_1$. Thus $\sigma(M_{m,\varphi,\varphi})=\{1\}$. 
\end{exm}

\section*{Acknowledgments}

This work was partially supported by the PRIN 2017 project ``Qualitative and quantitative aspects of nonlinear PDEs''.

\vspace*{0.5cm}
\begin{center}
	\textsc{Rosario Corso, Dipartimento di Matematica e Informatica} \\
	\textsc{Università degli Studi di Palermo, I-90123 Palermo, Italy} \\
	{\it E-mail address}: {\bf rosario.corso02@unipa.it}
\end{center}

\end{document}